\date{January 21, 2008}

\documentclass[reqno,12pt]{amsart}
\usepackage{latexsym,amsmath,amsfonts,amscd,amssymb}

\theoremstyle{plain}
\newtheorem{theorem}{Theorem}[section]
\newtheorem{corollary}[theorem]{Corollary}
\newtheorem{lemma}[theorem]{Lemma}
\newtheorem{proposition}[theorem]{Proposition}
\theoremstyle{definition}
\newtheorem{definition}[theorem]{Definition}
\newtheorem{remark}[theorem]{Remark}
\numberwithin{equation}{section}

\newcommand{\s}{\sigma}

\newcommand{\cM}{\mathcal{M}}

\newcommand{\cJ}{\mathcal{J}}

\newcommand{\cN}{\mathcal{N}}

\newcommand{\cO}{\mathcal{O}}

\newcommand{\cS}{\mathcal{S}}
\newcommand{\cT}{\mathcal{T}}

\newcommand{\cX}{\mathcal{X}}

\newcommand{\QQ}{\mathbb{Q}}
\newcommand{\RR}{\mathbb{R}}
\newcommand{\CC}{\mathbb{C}}
\newcommand{\HH}{\mathbb{H}}
\newcommand{\PP}{\mathbb{P}}
\newcommand{\ZZ}{\mathbb{Z}}

\newcommand{\inc}{\hookrightarrow}
\newcommand{\lto}{\longrightarrow}

\newcommand{\too}{\longrightarrow}

\newcommand{\ox}{\otimes}

\newcommand{\la}{\langle}
\newcommand{\ra}{\rangle}
\newcommand{\frM}{{\frak M}}

\newcommand{\GCD}{\mathrm{gcd}}
\newcommand{\Jac}{\mathrm{Jac}}
\DeclareMathOperator{\rk}{rk} \DeclareMathOperator{\codim}{codim}
 
\DeclareMathOperator{\Hom}{Hom} \DeclareMathOperator{\Ext}{Ext}
 \DeclareMathOperator{\Pic}{Pic}
\DeclareMathOperator{\gr}{gr}
\DeclareMathOperator{\Gr}{Gr}

\newcommand{\scp}{{\s_c^+}}
\newcommand{\smp}{{\s_m^+}}
\newcommand{\scm}{{\s_c^-}}

\title{Torelli theorem for the moduli spaces of pairs}


\author{Vicente Mu\~noz}
  \address{Instituto de Ciencias Matem{\'a}ticas CSIC-UAM-UCM-UC3M \\
  Consejo Superior de Investigaciones Cient{\'\i}ficas \\ Serrano 113 bis
  \\ 28006 Madrid \\ Spain}
  \address{Facultad de Matem\'{a}ticas \\ Universidad Complutense
  de Madrid \\ Plaza Ciencias 3
  \\ 28040 Madrid \\ Spain}
  \email{vicente.munoz@imaff.cfmac.csic.es}

\thanks{Partially supported through grant MEC
(Spain) MTM2007-63582}

\begin{document}
\maketitle

\begin{abstract}
 Let $X$ be a smooth projective curve of genus $g\geq 2$ over $\CC$.
 A pair $(E,\phi)$ over $X$ consists of an algebraic vector bundle
 $E$ over $X$ and a section $\phi \in H^0(E)$.
 There is a concept of stability for pairs which depends on a
 real parameter $\tau$. Here we prove that the third cohomology
 groups of the moduli spaces of $\tau$-stable pairs with fixed
 determinant 
 and rank $n\geq 2$
 are polarised pure Hodge structures, and
 they are isomorphic to $H^1(X)$ with its natural polarisation
 (except in very few exceptional cases).
 This implies a Torelli theorem for such moduli spaces. We
 recover that the third cohomology group of the moduli space of
 stable bundles of rank $n\geq 2$ and fixed determinant
 is a polarised pure Hodge structure, which is
 isomorphic to $H^1(X)$. We also prove Torelli theorems for the
 corresponding moduli spaces of pairs and bundles with non-fixed
 determinant.
\end{abstract}

\section{Introduction} \label{sec:introduction}

Let $X$ be a smooth projective curve of genus $g\geq 2$ over the
field of complex numbers. Fix $n \geq 2$ and $d\in\ZZ$. We shall
denote by $M_X(n,L_0)$ the moduli space of polystable bundles $E$
over $X$ of rank $n$ and determinant $\det(E)=L_0$, where $L_0$ is
a line bundle of degree $d$. This is a projective variety, which
is smooth over the locus of stable bundles. If $n$ and $d$ are
coprime, then there are no properly semistable bundles, and
$M_X(n,L_0)$ is smooth and projective. If $n$ and $d$ are not
coprime, then the open subset of stable bundles
$M^s_X(n,L_0)\subset M_X(n,L_0)$ is a smooth quasi-projective
variety, and $M_X(n,L_0)$ is in general singular.

A pair $(E,\phi)$ over $X$ consists of a bundle $E$ of rank $n$
and determinant $\det(E)=L_0$ over $X$ together with a section
$\phi\in H^0(E)$. There is a concept of stability for a pair which
depends on the choice of a parameter $\tau \in \RR$. This gives a
collection of moduli spaces of $\tau$-polystable pairs
$\frM_X(\tau;n,L_0)$, which are projective varieties. It contains
a smooth open subset $\frM_X^s(\tau;n,L_0)\subset
\frM_X(\tau;n,L_0)$ consisting of $\tau$-stable pairs. Pairs are
discussed at length in \cite{B,BD,GP,MOV1}.

The range of the parameter $\tau$ is an open interval
$I=I_{n,d}=(\tau_m,\tau_M) \subset \RR$. This interval is split by
a finite number of critical values $\tau_c$. For a non-critical
value $\tau\in I$, there are no properly semistable pairs, so
$\frM_X(\tau;n,L_0)=\frM_X^s(\tau;n,L_0)$ is smooth and
projective. For a critical value $\tau=\tau_c$,
$\frM_X(\tau;n,L_0)$ is in general singular at properly
$\tau$-semistable points.

Our first main result computes the third cohomology group of
$\frM_X^s(\tau;n,L_0)$.

\begin{theorem} \label{thm:H3-pairs}
Let $n\geq 2$. For $n=2$ let $\tau_L=\tau_M-1$, otherwise set
$\tau_L=\tau_M$, and let $\tau\in (\tau_m,\tau_L)$. Assume that we
are not in one of the following ``bad'' cases:
  \begin{itemize}
  \item $(n,g,d)= (3,2,2k)$, $k=1,2,3$;
  \item $(n,g,d)=(2,2,5),(2,3,5),(2,2,6)$, $\tau= \tau_M-2$;
  \item $(n,g,d)=(3,2,5)$, $\tau= 2$.
 \end{itemize}
Then
 \begin{enumerate}
 \item $H^3(\frM_X^s(\tau;n,L_0))$ is a pure Hodge structure which is naturally
 polarised.
 \item There is an isomorphism $H^3(\frM_X(\tau;n,L_0)) \cong H^1(X)$
 of polarised Hodge structures.
 \end{enumerate}
\end{theorem}

For $n=2$, the moduli space $\frM_X(\tau;n,L_0)$ for $\tau\in
(\tau_L,\tau_M)$ is a projective space. Therefore the above result
for the third cohomology does not hold. On the other hand, the
special cases that we remove are of low genus, rank and degree,
and for particular critical values of $\tau$.

The following corollary is a Torelli theorem for the moduli spaces
of $\tau$-stable pairs.

\begin{corollary} \label{cor:Torelli-fixed-det}
 Let $X$ be a smooth projective curve, $n\geq 2$, $L_0$ a line bundle
 of degree $d$ over $X$. For $n=2$ let $\tau_L=\tau_M-1$, otherwise set
 $\tau_L=\tau_M$, and let $\tau\in (\tau_m,\tau_L)$. Consider another
 collection $X'$, $n'$, $L_0'$, $d'$ and $\tau'$.
 Assume that $(n,g,d,\tau)$ and $(n',g',d',\tau')$ are not
 one of the exceptional cases enumerated in Theorem \ref{thm:H3-pairs}. Then
\begin{itemize}
 \item If
 $\frM_X(\tau;n,L_0)$ and $\frM_{X'}(\tau';n',L_0')$ are isomorphic
 algebraic varieties, then $X\cong X'$.
 \item If
 $\frM_X^s(\tau;n,L_0)$ and $\frM_{X'}^s(\tau';n',L_0')$ are isomorphic
 algebraic varieties, then $X\cong X'$.
\end{itemize}
\end{corollary}

The first statement is reduced to the second one since
$\frM_X^s(\tau;n,L_0)$ is the smooth locus of
$\frM_X(\tau;n,L_0)$. The second one is proved by looking at the
polarised Hodge structure $H^3(\frM_X^s(\tau;n,L_0))\cong H^1(X)$
and recovering $X$ from $H^1(X)$ via the usual Torelli theorem.

For values of $\tau$ slightly bigger than $\tau_m$, there is a
natural map $\frM_X(\tau;n,L_0)\to M_X(n,L_0)$. This allows to
prove the following result.

\begin{theorem} \label{thm:H3-bundles}
Let $n\geq 2$. Assume that $(n,g,d)\neq (2,2,2k),
(3,2,3k),(2,3,2k)$, $k\in \ZZ$. Then
 \begin{enumerate}
 \item $H^3(M_X^s(n,L_0))$ is a pure Hodge structure which is naturally polarised.
 \item There is an isomorphism $H^3(M_X^s(n,L_0)) \cong H^1(X)$ of
 polarised Hodge structures.
 \end{enumerate}
\end{theorem}

For $(n,g,d)=(2,2,even)$, the moduli space $M_X(n,L_0)$ is
isomorphic to $\PP^3$ (see \cite{NR1}), hence the above result
does not hold. Also, $M_X^s(2,L_0)=M_X(2,L_0)-S$, where $S=\Jac^d
X/\pm1$. From this it is easy to see that the result does not hold
for $M_X^s(2,L_0)$ either.

For  $(n,g,d)=(2,3,2k), (3,2,3k)$, $k\in\ZZ$, $H^3(M_X^s(n,L_0))$
is a mixed Hodge structure, whose graded piece $\Gr_W^3
H^3(M_X^s(n,L_0))$ is isomorphic to $H^1(X)$. However we are not
able to polarise it with the methods in this paper.

\begin{corollary} \label{cor:torelli-bundles}
Let $X$ be a curve, $(n,g,d)\neq (2,2,2k), (3,2,3k), (2,3,2k)$,
$k\in\ZZ$. If $X'$ is another curve, $n'\geq 2$, and $L_0'$ is a
line bundle over $X'$ of degree $d'$, and $(n',g',d')\neq
(2,2,2k), (3,2,3k), (2,3,2k)$, $k\in\ZZ$, then a Torelli theorem
holds:
 \begin{itemize}
 \item If $M_X(n,L_0)$ and $M_{X'}(n',L_0')$ are isomorphic
 algebraic varieties, then $X\cong X'$.
\item If
 $M_X^s(n,L_0)$ and $M_{X'}^s(n',L_0')$ are isomorphic
 algebraic varieties, then $X\cong X'$.
 \end{itemize}
\end{corollary}

When $n$ and $d$ are coprime, Theorem \ref{thm:H3-bundles} has
been proved in \cite{NR2,Tyu,Mum-New}. In the non-coprime case,
the Torelli theorem was proved in \cite{Kou-Pan} by different
methods. Theorem \ref{thm:H3-bundles} has been proved by Arapura
and Sastry \cite{Ara-Sas}, but under the condition
$g>\frac{3}{n-1}+\frac{n^2+3n+1}{2}$. Here we remove this lower
bound assumption.

\medskip

Our strategy of proof is the following. First, we find more
convenient to rephrase the problem in terms of triples. A triple
$(E_1,E_2,\phi)$ consists of a pair of bundles $E_1$, $E_2$ of
ranks $n_1,n_2$, with $\det(E_1)=L_1$, $\det(E_2)=L_2$,
respectively, over $X$ and a homomorphism $\phi:E_2\to E_1$. Here
$L_1$, $L_2$ are fixed line bundles of  degrees $d_1,d_2$,
respectively. There is a suitable concept of stability for triples
depending on a real parameter $\sigma$. This gives rise to moduli
spaces $\cN_X(\sigma; n_1,n_2,L_1,L_2)$ of $\s$-polystable
triples.

There is an identification of moduli spaces of pairs and triples
given by
 $$
 \frM_X(\tau;n,L_0)\to \cN_X(\s;n,1,L_0,\cO), \qquad
 (E,\phi)\mapsto (E,\cO,\phi),
 $$
where $\cO$ is the trivial line bundle, and $\s= (n+1) \tau - d$.
Actually, this rephrasing is a matter of aesthetic. The arguments
can be carried out directly with the moduli spaces of pairs, but
the formulas which appear using triples look more symmetric, and
clearly they could eventually be generalised to the case of
triples of arbitrary ranks $n_1,n_2$.

The range of the parameter $\s$ is an interval $I=(\s_m,\s_M)
\subset \RR$ split by a finite number of critical values $\s_c$.
When $\s$ moves without crossing a critical value, then
$\cN_{X,\s}=\cN_X(\s;n,1,L_1,L_2)$ remains unchanged, but when
$\s$ crosses a critical value, $\cN_{X,\s}$ undergoes a birational
transformation which we call \emph{a flip}. We compute the
codimension of the locus where this birational map is not an
isomorphism to be at least $2$, except in the bad case $n=2$,
$\s=\s_M-3$ (corresponding to $\tau=\tau_M-1$). This allows to
prove that the Hodge structures $H^3(\cN_{X,\s}^s)$ are identified
for different values of $\s\in I$. An explicit description of the
moduli space $\cN_{X,\s_M^-}$, where $\s_M^-=\s_M-\epsilon$,
$\epsilon>0$ small, allows to compute $H^3(\cN_{X,\s_M^-})$ by
induction on $n$.

For $\s=\s_m^+=\s_m+\epsilon$, $\epsilon>0$ small, we have a
morphism $\cN_{X,\sigma_m^+} \to M_X(n,L_0)$, $L_0=L_1\ox
L_2^{-n}$. This is a fibration over the locus $M^s_X(n,L_0)$, when
$d=\deg(L_0)$ is large enough. A computation of the codimension of
the locus of strictly semistable bundles allows to check that
$H^3(M_X^s(n,L_0))\cong H^3(\cN_{X,\s_m^+})$.

\bigskip

We end up with the study of the case of non-fixed determinant. Let
$M_X(n,d)$ denote the moduli space of semistable bundles of rank
$n$ and degree $d$ over $X$. The open subset consisting of stable
bundles will be denoted $M^s_X(n,d) \subset M_X(n,d)$. There is a
natural map $M_X(n,d) \to \Jac^d X$, whose fiber over $L_0$ is
$M_X(n,L_0)$. Also let $\frM_X(\tau;n,d)$ be the moduli space of
$\tau$-semistable pairs $(E,\phi)$, where $E$ is a bundle of rank
$n$ and degree $d$. There is a map $\frM_X(\tau ; n,d)\to \Jac^d
X$ as before. Denote by $\frM_X^s(\tau;n,d)$ the open subset of
$\tau$-stable triples. The following theorem is proved by reducing
to the case of fixed determinant.

\begin{corollary} \label{cor:Torelli-non-fixed-det}
The following Torelli theorems hold. Let $X$, $X'$ be projective
smooth curves of genus $g,g'\geq 2$. Let $n,n'\geq 2$, and
$d,d'\in\ZZ$.

Let $\tau\in (\tau_m,\tau_L)$, $\tau'\in (\tau_m',\tau_L')$.
Assume that $(n,g,d,\tau)$ and $(n',g',d',\tau')$ are not in one
of the bad cases enumerated in Theorem \ref{thm:H3-pairs}. Then
 \begin{itemize}
 \item If $\frM_X(\tau;n,d)\cong \frM_{X'}(\tau';n',d')$ then $X\cong X'$.
 \item If $\frM_X^s(\tau;n,d)\cong \frM_{X'}^s(\tau';n',d')$ then $X\cong X'$.
\end{itemize}

Assume that $(n,g,d), (n'g',d')$ are not of the form $(2,2,2k),
(3,2,3k), (2,3,2k)$, $k\in\ZZ$. Then
\begin{itemize}
 \item If $M_X(n,d)\cong M_{X'}(n',d')$ then $X\cong X'$.
 \item If $M_X^s(n,d)\cong M_{X'}^s(n',d')$ then $X\cong X'$.
 \end{itemize}
\end{corollary}

\section{Moduli spaces of triples} \label{sec:triples}

Let $X$ be a smooth projective curve of genus $g\geq 2$ over $\CC$.
A triple $T = (E_{1},E_{2},\phi)$ on $X$ consists of two vector
bundles $E_{1}$ and $E_{2}$ over $X$, of ranks $n_1$ and $n_2$ and
degrees $d_1$ and $d_2$, respectively, and a homomorphism $\phi
\colon E_{2} \to E_{1}$. We shall refer to $(n_1,n_2,d_1,d_2)$ as
the \emph{type} of the triple.

For any $\s \in \RR$ the $\s$-slope of $T$ is defined by
 $$
   \mu_{\s}(T)  =
   \frac{d_1+d_2}{n_1+n_2} + \s \frac{n_{2}}{n_{1}+n_{2}}\ .
 $$
We say that a triple $T = (E_{1},E_{2},\phi)$ is $\s$-stable if
$\mu_{\s}(T') < \mu_{\s}(T)$ for any proper subtriple $T' =
(E_{1}',E_{2}',\phi')$. We define $\s$-semistability by replacing
the above strict inequality with a weak inequality. A triple $T$ is
$\s$-polystable if it is the direct sum of $\s$-stable triples of
the same $\s$-slope. We denote by
  $$
  \cN_X(\s;n_1,n_2,d_1,d_2)
  $$
the moduli space of $\s$-polystable triples of type
$(n_1,n_2,d_1,d_2)$. This moduli space was constructed in \cite{BGP}
and \cite{Sch}. It is a complex projective variety. The open subset
of $\s$-stable triples will be denoted by
$\cN_X^s(\s;n_1,n_2,d_1,d_2)$.

Let $L_1,L_2$ be two bundles of degrees $d_1,d_2$ respectively.
Then the moduli spaces of $\sigma$-semistable triples
$T=(E_1,E_2,\phi)$ with $\det(E_1)=L_1$ and $\det(E_2)=L_2$ will
be denoted
  $$
  \cN_X(\s;n_1,n_2,L_1,L_2)\, ,
  $$
and $\cN_X^s(\s;n_1,n_2,L_1,L_2)$ will be the open subset of
$\s$-stable triples.

\medskip

Let $\mu(E)=\deg(E)/\rk(E)$ denote the slope of a bundle $E$, and
let $\mu_i=\mu(E_i)=d_i/n_i$, for $i=1,2$. Write
  \begin{align*}
  \s_m = &\, \mu_1-\mu_2\ ,  \\
  \s_M = & \left\{ \begin{array}{ll}
    \left(1+ \frac{n_1+n_2}{|n_1 - n_2|}\right)(\mu_1 - \mu_2)\ ,
      \qquad & \mbox{if $n_1\neq n_2$\ ,} \\ \infty, & \mbox{if $n_1=n_2$\
      ,}
      \end{array} \right.
  \end{align*}
and let $I$ denote the interval $I=(\s_m,\s_M)$. Then a necessary
condition for $\cN_X^s(\s;n_1,n_2,d_1,d_2)$ to be non-empty is that
$\s\in I$ (see \cite{BGPG}). Note that $\s_m>0$. To study the
dependence of the moduli spaces on the parameter $\s$, we need to
introduce the concept of critical value \cite{BGP,MOV1}.

\begin{definition}\label{def:critical}
The values of $\s_c\in I$ for which there exist $0 \le n'_1 \leq
n_1$, $0 \le n'_2 \leq n_2$, $d'_1$ and $d'_2$, with $n_1'n_2\neq
n_1n_2'$, such that
 \begin{equation}\label{eqn:sigmac}
 \s_c=\frac{(n_1+n_2)(d_1'+d_2')-(n_1'+n_2')(d_1+d_2)}{n_1'n_2-n_1n_2'},
 \end{equation}
are called \emph{critical values}.
\end{definition}

The interval $I$ is split by a finite number of values $\s_c \in I$.
The stability and semistability criteria  for two values of $\s$
lying between two consecutive critical values are equivalent; thus
the corresponding moduli spaces are isomorphic. When $\s$ crosses a
critical value, the moduli space undergoes a transformation which we
call a \emph{flip}. We shall study the flips in some detail in the
next section.

\subsection*{Relationship with pairs}

A pair $(E,\phi)$ over $X$ consists of a vector bundle $E$ of rank
$n$ and with $\det(E)=L_0$, where $L_0$ is some fixed bundle of
degree $d$, and $\phi\in H^0(E)$. Let $\tau\in \RR$. We say that
$(E,\phi)$ is $\tau$-stable (see \cite[Definition 4.7]{GP}) if:
 \begin{itemize}
 \item For any subbundle $E'\subset E$, we have $\mu(E')<\tau$.
 \item For any subbundle $E'\subset E$ with $\phi\in H^0(E')$, we
 have $\mu(E/E')>\tau$.
 \end{itemize}
The concept of $\tau$-semistability is defined by replacing the
strict inequalities by weak inequalities. A pair $(E,\phi)$ is
$\tau$-polystable if $E=E'\oplus E''$, where $\phi\in H^0(E')$ and
$E''$ is a polystable bundle of slope $\tau$. The moduli space of
$\tau$-polystable pairs is denoted by $\frM_X(\tau; n, L_0)$.

Interpreting $\phi\in H^0(E)$ as a morphism $\phi:\cO \to E$,
where $\cO$ is the trivial line bundle on $X$, we have a map
$(E,\phi)\mapsto (E,\cO,\phi)$ from pairs to triples. The
$\tau$-stability of $(E,\phi)$ corresponds to the $\s$-stability
of $(E,\cO,\phi)$, where (see \cite{BGP})
 \begin{equation}\label{eqn:s-to-tau}
 \s=(n+1)\tau -d .
 \end{equation}
Therefore we have an isomorphism of moduli spaces
   \begin{equation}\label{eqn:isom}
   \frM_X(\tau;n,L_0) \cong \cN_X(\sigma;n,1,L_0,\cO)\, ,
   \end{equation}
Alternatively, (\ref{eqn:isom}) may be taken as the definition of
the moduli space of pairs. Note that $\s_m$ and $\s_M$ correspond
under (\ref{eqn:s-to-tau}) to
 $$
 \begin{aligned}
 \tau_m &\,=\frac{d}{n}, \\
 \tau_M&\, =\frac{d}{n-1}.
 \end{aligned}
 $$

\begin{theorem}\label{thm:pairs}
 For non-critical values $\s\in I$, $\cN_{X,\s}=\cN_X(\sigma;n,1,L_1,L_2)$ is smooth and
 projective, and it only consists of $\s$-stable points (i.e.
 $\cN_{X,\s}=\cN_{X,\s}^s$). For critical values $\s=\s_c$,
 $\cN_{X,\s}$ is projective, and the open subset
 $\cN_{X,\s}^s\subset \cN_{X,\s}$ is smooth. The dimension of
 $\cN_{X,\s}$ is $(n^2-n-1)(g-1)+d_1-nd_2-1$.
\end{theorem}

\begin{proof}
  In general, if $\s$ is not a critical value
  for triples of type $(n_1,n_2,d_1,d_2)$ and $\GCD(n_1,n_2,d_1+d_2) = 1$,
  then $\s$-semistability is equivalent to $\s$-stability. This
  follows from \cite[Remark 3.8]{MOV1}.

  Smoothness for the $\s$-stable points follows from
  \cite[Proposition 6.3]{BGP},
  since any $\s$-stable triple $T=(E_1,E_2,\phi)$ of type
  $(n,1,d_1,d_2)$ satisfies automatically that $\phi:E_2\to E_1$
  is injective. The result if for non-fixed determinant, but the
  proof carries over to the case of fixed determinant.

  The dimension appears in \cite[Theorem 5.13]{GP} in the case of
  non-fixed determinant. Going over the proof, we see that we only
  have to substract $2g$ to the formula in \cite[Theorem 5.13]{GP}.
\end{proof}

There is an isomorphism
 \begin{equation}\label{eqn:lll}
 \cN_X(\s; n,1,L_1,L_2) \cong \cN_X(\s; n,1,L_1\ox (L_2^*)^{\ox n},\cO),
 \end{equation}
given by $(E_1,L_2,\phi)\mapsto (E_1\ox L_2^*, \cO, \phi)$, so the
moduli space (\ref{eqn:isom}) is as general as the moduli spaces
$\cN_X(\s; n,1,L_1,L_2)$.

\section{Flips for the moduli spaces of pairs} \label{sec:flips}

The homological algebra of triples is controlled by the
hypercohomology of a certain complex of sheaves which appears when
studying infinitesimal deformations \cite[Section 3]{BGPG}. Let
$T'=(E'_1,E'_2,\phi')$ and $T''=(E''_1,E''_2,\phi'')$ be two
triples of types $(n_{1}',n_{2}',d_{1}',d_{2}')$ and
$(n_{1}'',n_{2}'',d_{1}'',d_{2}'')$, respectively. Let
$\Hom(T'',T')$ denote the linear space of homomorphisms from $T''$
to $T'$, and let $\Ext^1(T'',T')$  denote the linear space of
equivalence classes of extensions of the form
 $$
  0 \lto T' \lto T \lto T'' \lto 0,
 $$
where by this we mean a commutative  diagram
  $$
  \begin{CD}
  0@>>>E_1'@>>>E_1@>>> E_1''@>>>0\\
  @.@A\phi' AA@A \phi AA@A \phi'' AA\\
  0@>>>E'_2@>>>E_2@>>>E_2''@>>>0.
  \end{CD}
  $$
To analyze $\Ext^1(T'',T')$ one considers the complex of sheaves
 \begin{equation} \label{eqn:extension-complex}
    C^{\bullet}(T'',T') \colon ({E_{1}''}^{*} \otimes E_{1}') \oplus
  ({E_{2}''}^{*} \otimes E_{2}')
  \overset{c}{\lto}
  {E_{2}''}^{*} \otimes E_{1}',
 \end{equation}
where the map $c$ is defined by
 $$
 c(\psi_{1},\psi_{2}) = \phi'\psi_{2} - \psi_{1}\phi''.
 $$

We introduce the following notation:
\begin{align*}
  \HH^i(T'',T') &= \HH^i(C^{\bullet}(T'',T')), \\
  h^{i}(T'',T') &= \dim\HH^{i}(T'',T'), \\ 
  \chi(T'',T') &= h^0(T'',T') - h^1(T'',T') + h^2(T'',T'). 
\end{align*}

By \cite[Proposition 3.1]{BGPG}, there are natural isomorphisms
  \begin{align*}
    \Hom(T'',T') &\cong \HH^{0}(T'',T'), \\
    \Ext^{1}(T'',T') &\cong \HH^{1}(T'',T').
  \end{align*}

We shall use the following results later:

\begin{lemma}[{\cite[Lemma 3.10]{Mu}}] \label{lem:H2=0}
 If $T''=(E_1'',E_2'',\phi'')$ is an injective triple, that is $\phi'':E_2''\to E_1''$ is
 injective, then $\HH^2(T'',T')=0$.
\end{lemma}

\begin{proposition}[{\cite[Proposition 3.2]{BGPG}}]
  \label{prop:chi(T'',T')}
  For any holomorphic triples $T'$ and $T''$ we have
  \begin{align*}
    \chi(T'',T') &=
    (1-g)(n''_1 n'_1 + n''_2 n'_2 - n''_2 n'_1) + n''_1 d'_1 - n'_1 d''_1
    + n''_2 d'_2 - n'_2 d''_2
    - n''_2 d'_1 + n'_1 d''_2.
  \end{align*}
\end{proposition}

\bigskip

Fix the type $(n_1,n_2,d_1,d_2)$ for the moduli spaces of triples.
For brevity, write $\cN_\s=\cN_X(\s;n_1,n_2,L_1,L_2)$. Let $\s_c\in
I$ be a critical value and set
 $$
 \scp = \s_c + \epsilon,\quad \scm = \s_c -
 \epsilon,
 $$
where $\epsilon > 0$ is small enough so that $\s_c$ is the only
critical value in the interval $(\scm,\scp)$.

\begin{definition}\label{def:flip-loci}
We define the \textit{flip loci} as
 \begin{align*}
 \cS_{\scp} &= \{ T\in\cN_{\scp}^s \ ;
 \ \text{$T$ is $\scm$-unstable}\} \subset\cN_{\scp}^s \ ,\\
 \cS_{\scm} &= \{ T\in\cN_{\scm}^s \ ;
 \ \text{$T$ is $\scp$-unstable}\}
 \subset\cN_{\scm}^s \ .
 \end{align*}
\end{definition}

It follows that (see \cite[Lemma 5.3]{BGPG})
 $$
 \cN_{\scp}^s-\cS_{\scp}=\cN_{\s_c}^s=\cN_{\scm}^s-\cS_{\scm}.
 $$

\begin{definition} \label{def:S-plusminus} Let $\sigma_c\in I$ be a
critical value given by $(n_1',n_2',d_1',d_2')$ in
(\ref{eqn:sigmac}), and let $(n_1'',n_2'',d_1'',d_2'') =
(n_1-n_1',n_2-n_2',d_1-d_1',d_2-d_2')$.
 \begin{itemize}
 \item[(1)] Define $\tilde{\mathcal{S}}_{\sigma_c^+}^0(n_1',n_2',d_1',d_2')$
   to be the set of all isomorphism classes of extensions
      \begin{displaymath}
      0 \lto T' \lto T \lto T'' \lto 0,
      \end{displaymath}
   where $T'$ and $T''$ are $\sigma_c^+$-stable triples with types
   $(n_1',n_2',d_1',d_2')$ and $(n_1'',n_2'',d_1'',d_2'')$ respectively,
   and for which $T$ is $\sigma_c^+$-stable, $T\in \cN_{\s_c^+}^s$.
   Note that in this case
   $\mu_{\s_c}(T')=\mu_{\s_c}(T)=\mu_{\s_c}(T'')$, and
   $\frac{n'_2}{n'_1+n'_2}<\frac{n''_2}{n''_1+n''_2}$.
 \item[(2)]  Define
      \begin{align*}
      \tilde{\mathcal{S}}^0_{\sigma_c^+} = \bigcup
      \tilde{\mathcal{S}}^0_{\sigma_c^+}(n_1',n_2',d_1',d_2'),
      \end{align*}
   where the union is over all $(n'_1,n'_2,d'_1,d'_2)$ and
   $(n''_1,n''_2,d''_1,d''_2)$ such that the above conditions
   apply.
 \item[(3)] Similarly, define $\tilde{\mathcal{S}}^{0}_{\sigma_c^-}
   (n'_1,n'_2,d'_1,d'_2)$ and $\tilde{\mathcal{S}}^0_{\sigma_c^-}$,
   where now $\frac{n'_2}{n'_1+n'_2}>\frac{n''_2}{n''_1+n''_2}$.
\end{itemize}
\end{definition}

The following is \cite[Lemma 5.8]{BGPG}. Actually the version in
\cite{BGPG} is stated for the case of non-fixed determinant, but the
fixed determinant version is completely similar.

\begin{lemma}\label{lem:vmaps}
There are maps $v^{\pm}:\tilde{\mathcal{S}}^0_{\sigma_c^{\pm}}
\longrightarrow\mathcal{N}^s_{\sigma_c^{\pm}}$ which map triples to
their equivalence classes. The images contain the flip loci
$\mathcal{S}_{\sigma_c^{\pm}}$.
\end{lemma}

\begin{proposition}\label{prop:codim-estimate}
Assume that $\HH^0(T'',T')=\HH^2(T'',T')=0$ for all
$\s_c^\pm$-stable triples $T'$, $T''$ of types
$(n_1',n_2',d_1',d_2')$, $(n_1'',n_2'',d_1'',d_2'')$ respectively.
Then $\mathcal{S}_{\sigma_c^{\pm}}\subset\mathcal{N}^s_{\sigma_c^{\pm}}$
are contained in subvarieties of codimension bounded below by
 $$
 \mathrm{min}\{-\chi(T',T'')\},
 $$
where the minimum is over all $(n_1',n_2',d_1',d_2')$ which
satisfy (\ref{eqn:sigmac}), $(n_1'',n_2'',d_1'',d_2'') =
(n_1-n_1',n_2-n_2',d_1-d_1',d_2-d_2')$ and
$\frac{n'_2}{n'_1+n'_2}<\frac{n''_2}{n''_1+n''_2}$ (in the case of
$\mathcal{S}_{\sigma_c^{+}}$) or
$\frac{n'_2}{n'_1+n'_2}>\frac{n''_2}{n''_1+n''_2}$ (in the case of
$\mathcal{S}_{\sigma_c^{-}}$).
\end{proposition}

\begin{proof}
 The proof of Proposition 5.10 in \cite{BGPG} goes over to this
 case. The condition $\s_c>2g-2$ in \cite[Proposition 5.10]{BGPG} is
 only needed to conclude the vanishing of $\HH^0(T'',T')$ and $\HH^2(T'',T')$.

 Fixing the determinant of the triples $T$ forces that, once $T'$
 is chosen, then the determinant of $T''$ is fixed. This reduces
 by $2g$ the dimension of the moduli space of $\s^\pm$-stable
 triples, and also reduces by $2g$ the dimension of the flip loci.
 Therefore the formula of the codimension is the same as in
 \cite[Proposition 5.10]{BGPG}.
\end{proof}

\section{Codimension estimates} \label{sec:codim-estimates}

We are going to apply Proposition \ref{prop:codim-estimate} to the
case of $n_2=1$, $n_1=n\geq 2$. Denote
$\cN_{X,\s}=\cN_X(\s;n,1,L_1,L_2)$. Here $L_1,L_2$ are line
bundles of degrees $d_1,d_2$ respectively. (We shall have no need
of particularising $L_2=\cO$ for the subsequent arguments to work,
so we will not do it).

We start by computing codimension estimates for the flip loci
$\cS_{\s_c^\pm}\subset \cN_{X,\s_c^\pm} = \cN_{X,\s_c^\pm}^s$.

\begin{proposition}\label{prop:Scmas}
  Suppose $n_2=1$, $n_1=n\geq 2$.
  Let $\s_c$ be a critical value with $\s_m < \s_c < \s_M$. Then
  \begin{itemize}
  \item $\codim \cS_{\s_c^+}\geq 3$, except in the case $n=2$, $g=2$,
  $d_1$ odd and $\s_c=\s_m+\frac32$ (in which case $\codim \cS_{\s_c^+}=2$).
  \item $\codim \cS_{\s_c^-}\geq 2$, except for $n=2$ and $\s_c=\s_M-3$
  (in which case $\codim \cS_{\s_c^-}=1$).
  Moreover, for $n=2$ we have that $\codim \cS_{\s_c^-}=2$
  only for $\s_c=\s_M-6$.
  \end{itemize}
\end{proposition}

\begin{proof}
Let us do the case of $\cS_{\sigma_c^+}$ first. The condition
 $$
 \frac{n'_2}{n'_1+n'_2}<\frac{n''_2}{n''_1+n''_2}
 $$
implies that $n_2'=0$ and $n_2''=1$. Since $T'$ and $T''$ are
$\sigma_c^+$-stable triples which are not isomorphic, we have that
$\HH^0(T'',T')=\Hom(T'',T')=0$. By Lemma \ref{lem:H2=0}, it is
clear that $\HH^2(T'',T')=0$. Proposition \ref{prop:chi(T'',T')}
gives (using that $n_2'=d_2'=0$, $n_2''=1$, and paying attention
to the fact that the roles of $T'$ and $T''$ are interchanged)
  $$
  -\chi(T',T'') =(g-1)n_1'n_1'' + n_1''d_1'-n_1'd_1''\, .
  $$
The equality $\mu_{\s_c}(T')=\mu_{\s_c}(T)$ is rewritten as
 $$
 \frac{d_1'}{n'_1} = \frac{d_1+d_2+\s_c}{n_1+1}\, .
 $$
Now $\s_c>\s_m=\frac{d_1}{n_1}-d_2$ implies that
 $$
 \frac{d_1'}{n'_1} > \frac{1}{n_1+1}\left( d_1+
 d_2+\frac{d_1}{n_1}-d_2 \right)=\frac{d_1}{n_1}\, .
 $$
So $\frac{d_1'}{n_1'}>\frac{d_1''}{n_1''}$, and hence
$n_1''d_1'-n_1'd_1''>0$. This implies
   $$
  -\chi(T',T'') \geq (g-1)n_1'n_1'' + 1 \geq 2\, ,
  $$
using that $n=n_1'+n_1''$, $0<n_1',n_1''<n$. Moreover, except in the
case $(n,g)=(2,2)$, we have that $-\chi(T',T'')\geq 3$. For $n=2$,
$g=2$, $n_1'=n_1''=1$, $-\chi(T',T'')=d_1'-d_1''+1$, with
$d_1'>d_1''$. For $d_1'-d_1''=1$, $d_1$ is odd, $d_1'=(d_1+1)/2$ and
$d_1''=(d_1-1)/2$. Therefore $\s_c=(d_1+3)/2-d_2=\s_m + 3/2$.

\medskip

Now we turn to the case of $\cS_{\s_c^-}$. The condition
 $$ 
 \frac{n'_2}{n'_1+n'_2}>\frac{n''_2}{n''_1+n''_2}
 $$ 
implies that $n_2'=1$ and $n_2''=0$. Since $T'$ and $T''$ are
$\sigma_c^-$-stable triples which are not isomorphic, we have that
$\HH^0(T'',T')=\Hom(T'',T')=0$. Lemma \ref{lem:H2=0} guarantees that
$\HH^2(T'',T')=0$. Proposition \ref{prop:chi(T'',T')} gives (using
that $n_2''=d_2''=0$, $n_2'=1$)
   $$
   -\chi(T',T'')=
   (g-1)n_1''(n_1'-1) + n_1''d_1'-n_1'd_1'' + d_1''-n_1''d_2\, .
   $$

Denote
 \begin{equation} \label{eqn:A}
 A= n_1''d_1'-n_1'd_1'' + d_1''-n_1''d_2\,.
 \end{equation}
We have
 $$
    A =n_1''(d_1-d_2) -d_1''(n_1'+n_1''-1)
     = n_1''(d_1-d_2)-d_1''(n_1-1)\, .
 $$
Also $\mu_{\s_c}(T'')=\mu_{\s_c}(T)$ means that
 $$
 \frac{d_1''}{n_1''}= \frac{d_1+d_2+\s_c}{n_1+1}\, ,
 $$
{}From where
 $$
 A= n_1''(d_1-d_2)-n_1''(n_1-1)\frac{d_1+d_2+\s_c}{n_1+1}
 \, .
 $$
Now
 $$
 \s_c < \s_M = \frac{2n_1}{n_1-1}\left(\frac{d_1}{n_1}-d_2\right)
 $$
so that
 $$
 A > n_1'' (d_1-d_2) - n_1'' \frac{n_1-1}{n_1+1} \left(d_1+d_2+
 \frac{2n_1}{n_1-1}\left(\frac{d_1}{n_1}-d_2\right)\right)=0.
 $$
This gives that
   $$
   -\chi(T',T'')= (g-1)n_1''(n_1'-1) + A \geq 2,
   $$
in the case $n_1'>1$.
In the case $n_1'=1$, we have a more explicit formula
 $$
 -\chi(T',T'')=A= (n_1-1)(d_1'-d_2)>0.
 $$
Hence $-\chi(T',T'')\geq 2$, for $n\geq 3$ and
for $n=2$ and $d_1'-d_2\geq 2$. The only remaining case
corresponds to $-\chi(T',T'')=d_1'-d_2=1$, $n_1'=n_1''=1$,
$d_1'=d_2+1$, $\s_M=2d_1-4d_2$, $\s_m=d_1/2-d_2$ and
$\s_c=3d_1''-d_1-d_2=\s_M-3$. Finally, note that for $n=2$,
$-\chi(T',T'')=2$ only in the case $d_1'-d_2=2$, which corresponds
to $\s_c=\s_M-6$.
\end{proof}

\begin{lemma} \label{lem:codim-triples}
 Let $\s_c$ be a critical value. Assume that $\s_c< \s_L$ in the
 case $n=2$. Then $\codim(\cN_{X,\s_c}-\cN_{X,\s_c}^s)\geq 5$ except in the following cases:
  \begin{itemize}
  \item $n=2$, $g=2,3$, $\s_c=\s_M-6$ and $d_1-2d_2=5$,
  \item $n=2$, $g=2$, $\s_c=\s_M-6$ and $d_1-2d_2=6$,
  \item $n=3$, $g=2$, $\s_c=2$ and $d_1-3d_2=4$,
  \item $n=3$, $g=2$, $\s_c=3$ and $d_1-3d_2=5$.
 \end{itemize}
\end{lemma}

\begin{proof}
  By Theorem \ref{thm:pairs}, the dimension of $\cN_{X,\sigma}$ is
   $$
   \dim \cN_X(\s; n,1,L_1,L_2)=(n^2-n-1)(g-1)+d_1-nd_2 -1 \, .
   $$
  The set $S=\cN_{X,\sigma_c}-\cN_{X,\sigma_c}^s$ is formed by strictly
  $\s_c$-polystable triples. Therefore it is covered by the
  images of the sets
  $$
  \cX_{L_1,L_2}\subset
  \cN_X^s(\s_c;n_1',1,d_1',d_2) \times M_X(n_1'',d_1'')\,,
  $$
  where $d_1=d_1'+d_1''$, $n_1=n_1'+n_1''$,
    \begin{equation}\label{eqn:1}
     \frac{d_1''}{n_1''}= \frac{d_1'+d_2+\s_c}{n_1'+1}  =
     \frac{d_1+d_2+\s_c}{n_1+1} \,,
    \end{equation}
  and $\cX_{L_1,L_2}$ corresponds to those triples of the form
  $T=(E_1,L_2,\phi)=(E_1',L_2,\phi)\oplus (E_1'',0,0)$ with fixed
  determinant $\det(E_1)=L_1$. Therefore $\cX_{L_1,L_2}$ is a fibration over
  $M_X(n_1'',d_1'')$ whose fibers are moduli spaces of $\s_c$-stable
  triples with fixed determinant $\det(E_1')=L_1\ox \det(E_1'')^{-1}$. Thus
   $$
   \dim \cX= (n_1'^2-n_1'-1)(g-1)+d_1'-n_1'd_2-1 + (n_1''^2(g-1)+1).
   $$
  So
   $$
   \begin{aligned}
   \codim \cX = & \,
   (n^2-n + n_1'-n_1'^2-n_1''^2)(g-1) +d_1-d_1' -(n-n_1')d_2-1 \\
    = & \, (2n_1'n_1''-n_1'')(g-1) + d_1'' - n_1''d_2-1 \\
    = &\, n_1''(2n_1'-1) (g-1) +d_1''- n_1''d_2-1 \, .
   \end{aligned}
   $$

Note that
  \begin{equation}\label{eqn:A+B}
  d_1''- n_1''d_2 = (n_1''d_1'-n_1'd_1'' + d_1''- n_1''d_2)+ (n_1'd_1''-n_1''d_1')\,.
  \end{equation}
Define $A= n_1''d_1'-n_1'd_1'' + d_1''-n_1''d_2$ as in
(\ref{eqn:A}). Using (\ref{eqn:1}), we get as in the proof of
Proposition \ref{prop:Scmas} that $A>0$. Also the inequality
$\s_c>\frac{d_1'}{n_1'}-d_2>0$ and (\ref{eqn:1}) give
  $$
  \frac{d_1''}{n_1''} > \frac{d_1' +d_2 +d_1'/n_1' - d_2}{n_1'+1}=\frac{d_1'}{n_1'} \, ,
  $$
hence $B=n_1'd_1''-n_1''d_1'>0$. To prove that $\codim \cX\geq 5$
we need to prove that
 $$
 1+\codim \cX= n_1''(2n_1'-1) (g-1) +A+B \geq 6.
 $$
This is true except possibly in the following cases:
  \begin{itemize}
  \item $n=2$. Then $n_1'=1$ and $n_1''=1$. We have that
  $A=d_1'-d_2$ and $B=d_1''-d_1'$. In this case we assume that
  $\s_c=3d_1''-d_1-d_2<\s_L=\s_M-3=2d_1-4d_2-3$, which can be
  rewritten as $A=d_1'-d_2>1$.
  So $1+\codim \cX=(g-1)+A+B\geq 6$
  except for $g=2,3$, $A=2$ and $B=1$ (that is, $\s_c=\s_M-6$ and
  $d_1-2d_2=5$), for $g=2$, $A=2$ and $B=2$ (that is, $\s_c=\s_M-6$ and
  $d_1-2d_2=6$).
  \item $n\geq 3$ and $n_1'=1$. Then $n_1''(2n_1'-1) (g-1)
  =n_1''(g-1)\geq 2$ and
  $A=n_1''(d_1'-d_2)\geq 2$. So $1+\codim \cX\geq 6$ except if
  $n_1''=2$, $g=2$, $d_1'-d_2=1$ and
  $B=d_1''-n_1''d_1'=1$. This means that
  $\s_c=2d_1''-d_1-d_2=3d_2-d_1+6$. As
  $\s_M=d_1-3d_2$, $\s_m=\frac13 (d_1-3d_2)$ and $\s_m<\s_c<\s_M$,
  we have that it must be $\s_c=2$ and $d_1-3d_2=4$.
  \item $n\geq 3$ and $n_1'\geq 2$. Then $n_1''(2n_1'-1) (g-1)\geq
  3$. So $1+\codim \cX\geq 6$ except if $n_1''(2n_1'-1) (g-1)=3$,
  $A=1$ and $B=1$. This means $n_1'=2$, $n_1''=1$, $g=2$,
  $A=d_1'-d_1''-d_2=1$, $B=2d_1''-d_1'=1$. So
  $\s_c=4d_1''-d_1-d_2= 3d_2-d_1+8$. As
  $\s_M=d_1-3d_2$, $\s_m=\frac13 (d_1-3d_2)$ and $\s_m<\s_c<\s_M$,
 we have that it must be $\s_c=3$ and $d_1-3d_2=5$.
  \end{itemize}
\end{proof}

Also we need codimension estimates for the families of properly
semistable bundles over $X$.

\begin{lemma} \label{lem:codimsemist}
Let $S$ be a bounded family of isomorphism classes of strictly
semistable bundles of rank $n$ and determinant $L_0$. Then $\dim
M_X(n,L_0)- \dim S \geq (n-1)(g-1)$.
\end{lemma}

\begin{proof}
We may stratify $S$ according to the ranks and degrees of the
elements in the Jordan-H\"older filtration of the bundles. So we
may assume that $S$ consists only of bundles whose Jordan-H\"older
filtration has associated graded object $\gr(Q)=\oplus_{i=1}^r
Q_i$, with $n_i=\rk Q_i$, $d_i=\deg Q_i$, $\bigotimes_{i=1}^r
\det(Q_i)=L_0$, $r\geq 2$. We use now \cite[Proposition
7.9]{BGPMN}, but it has to be modified to take into account that
we are fixing the determinant of $Q$. This fixes the determinant
of one of the $Q_i$, say $Q_r$. This reduces the dimension stated
in \cite[Proposition 7.9]{BGPMN} by $g$. So
 $$
 \dim S\leq \left( \sum n_i^2 +\sum_{i<j} n_in_j\right)(g-1) +1 -g
 \,.
 $$
As $\dim M_X(n,L_0)= n^2(g-1) +1 -g$, we have that
 $$
 \dim M(n,L_0)- \dim S \geq \sum_{i<j} n_in_j (g-1)\, .
 $$
The minimum of the right hand side is attained for $n=2$, $n_1=1$,
$n_2=n-1$, whence the statement.
\end{proof}

\begin{lemma} \label{lem:codimMX}
Suppose $n\geq 2$, and let $S=M_X(n,L_0)-M_X^s(n,L_0)$ be the locus
of strictly polystable bundles. Then $\dim M_X(n,L_0)- \dim S \geq
2(n-1)(g-1)-1$.
\end{lemma}

\begin{proof}
 Working as in the proof of Lemma \ref{lem:codimsemist}, we have
 that if $S$ consists of polystable bundles $Q=\oplus Q_i$, with
 $n_i=\rk Q_i$, $d_i=\deg Q_i$, $\bigotimes_{i=1}^r \det(Q_i)=L_0$, $r\geq
 2$, then
 $$
 \dim S= \left( \sum n_i^2  (g-1) +1 \right) -g
 \,.
 $$
As $\dim M_X(n,L_0)= n^2(g-1) +1 -g$, we have that
 $$
 \dim M_X(n,L_0)- \dim S \geq 2 \sum_{i<j} n_in_j (g-1) +1-r\geq 2(n-1)(g-1)-1\,
 ,
 $$
 since the minimum occurs for $r=2$, $n_1=1$, $n_2=n-1$.
\end{proof}

\section{Cohomology groups of the moduli spaces of pairs} \label{sec:H1-2-3}

Now we aim to compute the cohomology groups $H^i(\cN_{X,\s}^s)$,
for $i=1,2,3$. Note that for $\s$ non-critical,
$\cN_{X,\s}=\cN_{X,\s}^s$, so we are actually computing
$H^i(\cN_{X,\s})$. Moreover, in this case $\cN_{X,\s}$ is smooth
and projective, so these are automatically pure Hodge structures.
As a byproduct, we shall obtain the cohomology groups
$H^i(M^s_X(n,L))$, for $i=1,2,3$. For $n$ and $d$ coprime,
$M_X(n,L)=M^s_X(n,L)$, which is smooth and projective. In this
case, these cohomology groups are well-known, however we shall
recover them easily from our arguments. For $n$ and $d$ not
coprime, the cohomology groups $H^i(M_X^s(n,L))$ seem to be known
to experts, but it is difficult to locate them in the literature.

We start with a small lemma.

\begin{lemma} \label{lem:M-S}
 Let $M$ be a smooth projective variety and let $S\subset M$ be a
 closed subset in the Zariski topology. If either:
 \begin{itemize}
 \item[(1)] $\codim S\geq 3$, or
 \item[(2)] $\codim S=2$ and $H^3(M-S)$ is a pure Hodge structure,
 \end{itemize}
 then
  $$
  H^i(M) \cong H^i(M-S),
  $$
 for $i\leq 3$.
\end{lemma}

\begin{proof}
  Let $m=\dim_\CC M$. Using Poincar\'{e} duality, the statement of the
  lemma is equivalent to
  proving an isomorphism $H^{2m-i}_c(M)\cong H^{2m-i}_c(M-S)$, where
  $H^*_c$ stands for cohomology with compact support. This is
  clear for $i\leq 2$, since $S$ has real codimension at least
  $4$. For $i=3$, we have an exact sequence
  $$
  H^{2m-4}_c(S) \stackrel{\partial}{\lto}
  H^{2m-3}_c(M-S) \lto H^{2m-3}_c(M) \lto 0 \, .
  $$
  The group $H_c^{2m-4}(S)$ is generated by the irreducible components
  $S_i$ of dimension $m-2$ of $S$. To see this, let $N_1=\bigcup_{i\neq j}
  (S_i\cap S_j)$. Then let $S_i^o\subset S_i$ be the smooth locus of
  $S_i-N_1$, and consider $N_2=\bigcup (S_i-S_i^o)$. Then $N=N_1\cup N_2$
  is of positive codimension in $S$, so that $H_c^{2m-4}(S)\cong H_c^{2m-4}
  (S-N)=\bigoplus_i H_c^{2m-4}(S_i^o)$. Note that $H_c^{2m-4}(S)$ is a pure Hodge
  structure of weight $(m-2,m-2)$.
  As $\partial$ preserves the weight of the Hodge structure, and
  $H_c^{2m-3}(M-S)$ is assumed to be of pure type, then
  $\partial=0$. The result follows.
\end{proof}

\begin{remark}\label{rem:M-S}
 If $S\subset M$ is of codimension $2$ and $M$ is a smooth
 projective variety, then in general, $H^3(M-S)$ is a mixed Hodge
 structure. It has pieces of weight $2$ and $3$, and the piece of
 weight $3$ satisfies $\Gr_3^W H^3(M-S) \cong H^3(M)$.
\end{remark}

\begin{proposition} \label{prop:equalH3}
  Assume $n\geq 3$. Let $\s_c$ be a critical value, $\s_m <\s_c<\s_M$.
  Then
   $$
   H^i(\cN_{X,\sigma_c^+}) \cong H^i(\cN_{X,\sigma_c^-})\cong
   H^i(\cN_{X,\sigma_c}^s),
   $$
   for $i\leq 3$. So all Hodge structures $H^i(\cN_{X,\s}^s)$ are
   naturally isomorphic, for $i\leq 3$. \newline (For $\s$
   non-critical, we have that $\cN_{X,\s}=\cN_{X,\s}^s$, so we
   are actually talking about $H^i(\cN_{X,\s})$.)
\end{proposition}

\begin{proof}
  By Proposition \ref{prop:Scmas}, $\codim \cS_{\sigma_c^+} \geq 3$ and $\codim
  \cS_{\sigma_c^-} \geq 2$. Then Lemma \ref{lem:M-S} applied to $\cS_{\sigma_c^+}$
  implies that $H^3(\cN_{X,\sigma_c^+}) \cong H^3(\cN_{X,\sigma_c}^s)$, since
  $\cN_{X,\sigma}^s=\cN_{X,\sigma_c^+}-\cS_{\sigma_c^+}$. In particular,
  $H^3(\cN_{X,\sigma_c}^s)$ is
  a pure Hodge structure. Applying Lemma  \ref{lem:M-S} to
  $\cS_{\sigma_c^-}$, we have $H^3(\cN_{X,\sigma_c^-}) \cong H^3(\cN_{X,\sigma_c}^s)$.

\end{proof}

\begin{proposition} \label{prop:case-n=2}
  Assume $n=2$. The critical values are the
  numbers $\s_c=\s_M-3n$, for $0<n<(\s_M-\s_m)/3$, $n\in\ZZ$. Denote
  $\s_L=\s_M-3$. If $\s_L>\s_m$ then for any $\s\in (\s_m, \s_L)$, we have
  $H^1(\cN_{X,\sigma}^s)=0$, $H^2(\cN_{X,\sigma}^s)\cong
  \ZZ\oplus\ZZ$,
  and all Hodge structures $H^3(\cN_{X,\s}^s)$ are
   naturally isomorphic. Moreover,
   $$
   H^3(\cN_{X,\sigma}^s)\cong H^1(X),
   $$
   for any $\s \in (\s_m, \s_L)$, with the exception of $g=2$, $d_1-2d_2=5$ and
   $\s=\s_M-6=4$. \newline (For $\s$
   non-critical, we have that $\cN_{X,\s}=\cN_{X,\s}^s$, so we
   are actually talking about $H^i(\cN_{X,\s})$.)
\end{proposition}

\begin{proof}
  The collection of moduli spaces $\cN_{X,\sigma}$ for $n=2$ is described in
  detail in \cite{Th}.
  The critical values are given by \cite[Lemma 5.3]{MOV1}
 to be of the form $\s_c=3(d_1-d_2-n) -d_1-d_2= \s_M-3n$, with
 $n>0$ and the constraint $\s_m<\s_c<\s_M$ (this also follows
 easily from (\ref{eqn:sigmac})).

  The last moduli space $\cN_{X,\s_M^-}$ is a projective
  space $\PP$ (see \cite[(3.1)]{Th}, or argue as in the proof of Proposition
  \ref{prop:last-H3} below: in the discussion of the proof of Proposition
  \ref{prop:last-H3}, $F$ should be a fixed line bundle,
  since the determinant is fixed, so $\cN_{X,\s_M^-}=U$ is the
  projective space $\PP=\PP H^1(F^*\ox L)$).

By \cite[(3.4)]{Th}, there is an embedding $X \hookrightarrow
\PP$, given by $p\mapsto [\delta\left((F^*\ox L(p))_p\right)] \in
\PP$, where $\delta:H^0((F^*\ox L(p))_p)\to H^1(F^*\ox L)$ is the
connecting map associated to the exact sequence
 $$
  F^*\ox L \to F^*\ox L(p) \to (F^*\ox L(p))_p.
 $$
By \cite[(3.19)]{Th}, the moduli space $\cN_{X,\s_L^-}$, for
$\s_L=\s_M-3$, is the blow-up of $\PP$ along $X$. The usual
computation of the cohomology of a blow-up gives that
$H^3(\cN_{X,\s_L^-}) \cong H^1(E)\cong H^1(X)$, where $E$ is the
exceptional divisor, which is a projective bundle over $X$. Also
$H^2(\cN_{X,\s_L^-}) \cong H^2(\PP)\oplus \ZZ[E] \cong \ZZ\oplus
\ZZ$ and $H^1(\cN_{X,\s_L^-}) =0$.

Now we have to prove that
 $$
 H^i(\cN_{X,\s_c^+})\cong H^i(\cN_{X,\s_c^-})\cong H^i(\cN_{X,\s_c}^s),
 $$
for $i\leq 3$ and $\s_m<\s_c<\s_L$. In the case $g\neq 2$ (or in
the case $g=2$ and $d_1$ even or $\s_c\neq \s_m+\frac32$), it
follows from the fact that $\codim \cS_{\sigma_c^+} \geq 3$ and
$\codim \cS_{\sigma_c^-} \geq 2$, which follows in turn from
Proposition \ref{prop:Scmas}. This is enough to complete the proof
of the proposition.

For the exceptional case $(n,g)=(2,2)$, $d_1$ odd and
$\s_c=\s_m+\frac32$, we have the following cases:
\begin{itemize}
 \item If $d_1-2d_2=1$ then there are no flips. So there is no
 such $\s_c$ and nothing to prove. (Actually $\cN_{X,\s}$ is
 a projective space for all allowable values of $\s$.)
 \item If $d_1-2d_2=3$ then $\s_c=\s_m+\frac32=\s_M-3$. Again there is no
 such $\s_c$. (Note that
 $\cN_{X,\s_m^+}=\cN_{X,\s_L^-}$, so for all $\s\in (\s_m,\s_L)$ we have $\cN_{X,\sigma}
 =\cN_{X,\s_m^+}$ whose cohomology has been computed above.)
 \item If $d_1-2d_2\geq 7$ then $\s_m+\frac32<\s_M-6$. Then for
 $\s_c=\s_m+\frac32$, $\codim \cS_{\sigma_c^+}=2$ but $\codim \cS_{\sigma_c^-}\geq 3$ (see the last line in
 Proposition \ref{prop:codim-estimate}). Then
 $H^3(\cN_{X,\s_c^+})\cong H^3(\cN_{X,\s_c^-})$ as required.
 \item If $d_1-2d_2=5$ then $\s_c=\s_m+\frac32=\s_M-6$. For $\s\in
 (\s_M-6,\s_M-3)$, we have $\cN_{X,\s}=\cN_{X,\s_L^-}$ for which the
 cohomology is computed above. For $\s\in
 (\s_m,\s_M-6)$, we have that $\cN_{X,\s}=\cN_{X,\s_m^+}$. As
 $\mu_1-\mu_2>2$, so have that  $H^3(\cN_{X,\s_m^+}) \cong H^3(M_X(2,L_0))$,
 where $L_0=L_1\ox L_2^{-2}$ (see the proof of Theorem
 \ref{thm:H3-bundles-proof}, and note that $\deg(L_0)$ is odd).
 We can twist $L_0$ by
 a large power of a line bundle $\mu$ to arrange that $\deg(L_0\otimes \mu^{2k})$ is
 large. As $M_X(2,L_0)\cong M_X(2, L_0\ox \mu^{2k})$, we get that
 $H^3(\cN_{X,\s_m^+}(2,1,L_1,L_2))$ is isomorphic to
 $H^3(\cN_{X,\s_m^+}(2,1,L_1\ox \mu^{2k},L_2))$. We have seen
 already that such cohomology group is isomorphic to $H^1(X)$
 for $d_1-2d_2>>0$. So $H^3(\cN_{X,\s_c^-})\cong H^3(\cN_{X,\s_c^+})$
 in this case. \newline The argument fails exactly for the critical value $\s_c=\s_M-6=
 \s_m+\frac32$. However, Remark \ref{rem:M-S} implies that $H^3(\cN_{X,\s_c}^s)$
 is a mixed Hodge structure whose $\Gr^3_W$-piece is isomorphic to
 $H^1(X)$.
\end{itemize}
\end{proof}



Now we shall compute the cohomology groups of $\cN_{X,\s}^s$ and
$M^s_X(n,L_0)$ simultaneously. We will prove the following two
theorems.

\begin{theorem}\label{thm:H3-bundles-proof}
  Suppose $n\geq 2$, and $(n,g,d)\neq (2,2,even)$.
  Then
  \begin{itemize}
  \item $H^1(M_X^s(n,L_0))=0$,
  \item $H^2(M_X^s(n,L_0))\cong \ZZ$,
  \item $H^3(M_X^s(n,L_0)) \cong H^1(X)$, except in the cases
  $(n,g,d)=(3,2,3k)$ and $(n,g,d)=(2,3,2k)$, $k\in \ZZ$. In these
  cases, $H^3(M_X^s(n,L_0))$ is a mixed Hodge structure and
  $\Gr_W^3H^3(M_X^s(n,L_0))\cong H^1(X)$.
  \end{itemize}
  (Recall that when $n$ and $d$ are coprime, $M_X(n,L_0)=M_X^s(n,L_0)$.)
\end{theorem}

\begin{theorem} \label{thm:pairs-all-H3}
  Assume $n\geq 2$. Let $\s\in (\s_m,\s_M)$ if $n\geq 3$ and
  $\s\in (\s_m,\s_M-3)$ if $n=2$.
  Then
  \begin{itemize}
  \item $H^1(\cN_{X,\s}^s) = 0$,
  \item $H^2(\cN_{X,\s}^s) \cong \ZZ\oplus \ZZ$,
  \item $H^3(\cN_{X,\s}^s) \cong H^1(X)$,
  \end{itemize}
  except for the cases $(n,g,d_1-2d_2)=(2,2,5)$, $\s=\s_M-6=4$,
  and $(n,g,d_1-3d_2)=(3,2,2k)$, $k=1,2,3$.
\end{theorem}

We prove both Theorem \ref{thm:H3-bundles-proof} and
\ref{thm:pairs-all-H3} as follows. First we know that Theorem
\ref{thm:pairs-all-H3} is true for $n=2$ by Proposition
\ref{prop:case-n=2}. Then we prove Theorem
\ref{thm:H3-bundles-proof} for rank $n$ assuming Theorem
\ref{thm:pairs-all-H3} for the same rank $n$. Finally we prove
Theorem \ref{thm:pairs-all-H3} for rank $n\geq 3$ using Theorem
\ref{thm:H3-bundles-proof} for rank $n-1$.

\bigskip

\noindent {\em Proof of Theorem \ref{thm:H3-bundles-proof}.\/}
Since the moduli spaces $M_X(n,L_0)$ and $M_X(n,L_0\ox \mu^{n})$
are isomorphic via $E\mapsto E\ox \mu$, for any fixed line bundle
$\mu$, we may assume that the degree $d$ is large, say
$d>(2g-2)n$.

Fix $d_1=d$,  $d_2=0$, $L_1=L_0$ and $L_2=\cO$, and consider the
moduli spaces $\cN_{X,\s}=\cN_X(\s; n,1,L_0,\cO)$. The moduli
space $\cN_{X,\s}$ for the smallest possible values of the
parameter can be explicitly described. Let $\smp=\s_m+\epsilon$,
$\epsilon>0$ small enough. By \cite[Proposition 4.10]{MOV1}, there
is a morphism
 $$
 \pi:\cN_{X,\smp} 
 \to M_X(n,L_0)
 $$
which sends $T=(E,L,\phi)\mapsto E$. Let
$U=\pi^{-1}(M^s_X(n,L_0))$. By \cite[Proposition 4.10]{MOV1},
$\pi:U\to M_X^s(n,L_0)$ is a projective fibration whose fibers are
the projective spaces $\PP H^0(E)$, since $d_1/n-d_2 >2g-2$.
Therefore
 $$
 \begin{aligned}
 H^1(U) \cong & \, H^1(M^s_X(n,L_0)), \\ H^2(U)\cong & \,
 H^2(M^s_X(n,L_0))\oplus \ZZ, \\
 H^3(U) \cong & \, H^3(M^s_X(n,L_0)).
 \end{aligned}
 $$

Let us compute the cohomology groups of $U$. The complement
$S=\cN_{X,\smp}-U$ consists of triples $(E,L_2,\phi)$ where $E$ is
semistable. By Lemma \ref{lem:codimsemist}, the codimension of the
family of such bundles is at least $(n-1)(g-1)$. The fiber over
$E$ is contained in (but it may not be equal to) $\PP H^0(E 
)$. As $E$ is semistable and $d_1/n-d_2>2g-2$, this dimension is
constant. So $\codim S \geq (n-1)(g-1)$.

If  $(n,g,d)\neq (3,2,3k)$ and $(n,g,d)\neq(2,3,2k)$, $k\in \ZZ$,
then $\codim S\geq 3$. Then Lemma \ref{lem:M-S} implies that
$H^1(U)=H^1(\cN_{X,\smp})$, $H^2(U)=H^2(\cN_{X,\smp})$ and
$H^3(U)=H^3(\cN_{X,\smp})$. The result now follows from Theorem
\ref{thm:pairs-all-H3} for rank $n$.

If either $(n,g,d)=(3,2,3k)$ or $(n,g,d)=(2,3,2k)$, $k\in \ZZ$,
then $\codim S=2$ and we only know by Remark \ref{rem:M-S} that
$H^3(U)$ is a mixed Hodge structure with $\Gr_W^3 H^3(U)
=H^3(\cN_{X,\smp})$. \hfill $\Box$

\bigskip

\noindent {\em Proof of Theorem \ref{thm:pairs-all-H3}.\/} We
assume $n\geq 3$ since the case $n=2$ is covered by Proposition
\ref{prop:case-n=2}.

Using Proposition \ref{prop:equalH3}, we see that it is enough to
prove Proposition \ref{prop:last-H3}. \qquad $\Box$

\begin{proposition} \label{prop:last-H3}
  Assume $n\geq 3$, and assume that Theorem \ref{thm:H3-bundles-proof}
  holds for rank $n-1$.
  Let $\s_M^-=\s_M-\epsilon$, $\epsilon>0$ small enough.
  Then
  \begin{itemize}
  \item $H^1(\cN_{X,\s_M^-}) = 0$,
  \item $H^2(\cN_{X,\s_M^-}) \cong \ZZ\oplus \ZZ$,
  \item $H^3(\cN_{X,\s_M^-}) \cong H^1(X)$,
  \end{itemize}
  except for the case $(n,g,d_1-3d_2)=(3,2,2k)$, $k=1,2,3$.
\end{proposition}

\begin{proof}
 By Propositions 7.5 and 7.6 in \cite{BGPG}, the triples in $\cN_{X,\s_M^-}$ satisfy that
 $\phi:L_2\to E_1$ is injective with torsion-free cokernel. Let
 $F=E_1/\phi(L_2)$. Then there is a short exact sequence $L_2\to E_1\to F$,
 and $F$ is a semistable bundle. Moreover, the determinant of $F$ is
 fixed, since $\det(F)=\det(E_1)\ox L_2^{-1}=L_1\ox L_2^{-1}$. The
 extension is always non-trivial. Moreover, if $F$ is stable, then
 any non-zero extension gives rise to a $\s_M^-$-stable triple.

By \cite[Proposition 6.9]{BGPG}, the dimension $\dim H^1(F^*\ox
L_2)$ is constant, for $F$ as above. Let $U\subset \cN_{X,\s_M^-}$
be the open subset formed by those triples with $F$ a stable
bundle. Then there is a fibration
 $$
 U\to M_X^s(n-1,L_1\ox L_2^{-1})
 $$
whose fibers are projective spaces $\PP H^1(F^*\ox L_2)$.
Therefore
 $$
 \begin{aligned}
 H^1(U)= & \, 0, \\
 H^2(U) \cong &\, H^2(M_X^s(n-1,L_1\ox L_2^{-1}))
\oplus \ZZ\cong \ZZ\oplus \ZZ, \\
 H^3(U)\cong &\, H^3(M_X^s(n-1,L_1\ox L_2^{-1})).
 \end{aligned}
 $$

As the dimension $\dim H^1(F^*\ox L_2)$ is constant, the
codimension of $\cN_{X,\s_M^-}-U$ is at least the codimension of a
locus of semistable bundles. By Lemma \ref{lem:codimsemist}
applied to $M_X(n-1,L_1\ox L^{-1}_2)$, this is at least
$(n-2)(g-1)$. If $(n-1,g,d_1-d_2)\neq (2,2,2k)$, $(3,2,3k)$,
$(2,3,2k)$, then this codimension is at least three. So
$H^i(\cN_{X,\s_M^-})\cong H^i(U)$, for $i\leq 3$. Applying Theorem
\ref{thm:H3-bundles-proof} for rank $n-1$ we get the result.

If $(n-1,g,d_1-d_2)=(3,2,3k)$ or $(2,3,2k)$, then $\codim
(\cN_{X,\s_M^-}-U)=2$, so $H^3(\cN_{X,\s_M^-})\cong
\Gr_W^3H^3(U)$. But again Theorem \ref{thm:H3-bundles-proof} gives
us the result.

Suppose finally that  $(n,g,d_1-d_2)=(3,2,2k)$, $k\in \ZZ$. By
Proposition \ref{prop:equalH3}, $H^3(\cN_{X,\s_M^-})\cong
H^3(\cN_{X,\s_m^+})$.
By the proof of Theorem \ref{thm:H3-bundles-proof},
 $$
 H^3(\cN_{X,\s_m^+}) \cong \Gr_W^3 H^3(M_X^s(3,L_0)),
 $$
for $d_1/3-d_2>2g-2=2$, $L_0=L_1\ox L_2^{-3}$ (note that we are
not assuming that the right hand side is known). So assume
$d_1-3d_2>6$. Twist $L_0$ by a line bundle $\mu$ of degree $1$ so
to change $\deg(L_0)$ to $\deg(L_0\ox \mu^{3k})=\deg(L_0) + 3k$.
This allows to change the parity of $d_1-3d_2$. Therefore
$H^3(\cN_{X,\s_m^+})$ is independent of the parity of $d_1-d_2
\equiv d_1-3d_2 \pmod 2$. Since the case that $d_1-d_2$ is odd is
already known, the result follows.
\end{proof}

\section{Reconstructing the polarisation}\label{sec:polarisation}

We want to show that $H^3(M_X^s(n,L_0))$ and $H^3(\cN_{X,\s}^s)$
have natural polarisations, which make them into polarised Hodge
structures. The word ``natural'' means that they are constructed
in families.

\begin{proposition} \label{prop:polarisation-MX}
 Suppose $(n,g,d)\neq (2,2,2k)$, $(3,2,3k)$, $(2,3,2k)$.
 The Hodge structure $H^3(M_X^s(n,L_0))$
 is naturally polarised, and the isomorphism
 $H^3(M_X^s(n,L_0))\cong H^1(X)$ respects the polarisations.
\end{proposition}

\begin{proof}
For $M=M_X^s(n,L_0)$ the polarisation is constructed as follows
(see \cite[Section 8]{Ara-Sas}; a similar argument is in
\cite[Section 4]{BM}). Let $\overline M=M_X(n,L_0)$. Since
$H^2(M)=\ZZ$, we have that $\Pic(M)=\ZZ$, so there is a unique
ample generator of the Picard group. Take a general $(k-3)$-fold
hyperplane section $Z\subset M$, where $k=\dim M$. By Lemma
\ref{lem:codimMX}, $\codim(M_X(n,L_0)-M_X^s(n,L_0))\geq 4$, so $Z$
is smooth. Define
  \begin{equation} \label{eqn:otroo}
  \begin{aligned}
  H^3(M)\ox H^3(M)  &\,\too \ZZ\, ,\\
  \beta_1\ox \beta_2 &\, \mapsto \la \beta_1\cup\beta_2, [Z]\ra\,.
  \end{aligned}
  \end{equation}
This is a polarisation. This is proved as follows (see
\cite[Proposition 6.2.1]{Ara-Sas}): take a generic $(k-4)$-fold
hyperplane section $W\subset M$. As
$\codim(M_X(n,L_0)-M_X^s(n,L_0))\geq 5$, $W$ is smooth. By the
Lefschetz theorem \cite[Theorem 6.1.1]{Ara-Sas} applied to the
open smooth variety $M$, we have that $H^3(W)\cong H^3(M)$. Then
by hard Lefschetz, cupping with the hyperplane class gives an
isomorphism $H^3(W) \stackrel{\cong}{\too} H^5(W)\cong H^3(W)^*$.
This map coincides with (\ref{eqn:otroo}), which proves that it is
a non-degenerate pairing.

Let $\pi:\cX\to \cT$ be the family of curves of genus $g$ with no
automorphisms, and let $\theta$ be the standard polarisation on
$R^1\pi_*\underline\ZZ$ corresponding to the cup product. We
consider the universal Jacobian
 $$
 q:\cJ^d\to \cT,
 $$
that is, the family of Jacobians $\Jac^d X$, for $X\in \cT$. Then
there is a universal moduli space
 $$
 p:\cM_\cT (n,d) \to \cJ^d
 $$
which puts over any $(X,L_0)\in \cJ^d$ the moduli space
$M_X^s(n,L_0)$. If $\underline \ZZ$ denotes the local system over
$\cM_\cT(n,d)$, then $R^3p_* \underline\ZZ$ is the local system
over $\cJ^d$ whose fibers are the Hodge structures
$H^3(M_X^s(n,L_0))$. Let $\theta'$ be the natural polarisation on
$R^3p_*\underline \ZZ$ as defined above. There is an isomorphism
 $$
 R^3p_*\underline \ZZ \cong q^*R^1\pi_*\underline\ZZ.
 $$

The natural map $\cJ^d\to M_g$ is dominant. So \cite[Lemma
8.1.1]{Ara-Sas} implies that there exists an integer $m\neq 0$
such that $\theta'=m\,\theta$. As any polarisation is a unique
positive multiple of a primitive polarisation, $\theta'$
determines a unique primitive polarisation on $H^3(M_X^s(n,L_0))$
for any $(X,L_0)$. So the isomorphism $H^3(M_X^s(n,L_0))\cong
H^1(X)$ respects the polarisations.
\end{proof}

\begin{proposition} \label{prop:polarisation-NX}
  Let $n\geq 2$. Let $\s\in (\s_m,\s_M)$ if $n\geq 3$ and
  $\s\in (\s_m,\s_M-3)$ if $n=2$.
 Assume that we are not in any of the cases enumerated in Lemma
 \ref{lem:codim-triples}, and also that $(n,g,d_1-3d_2)\neq (3,2,2k)$,
 $k=1,2,3$. Then
 $H^3(\cN_{X,\s}^s)$ is naturally polarised, and the isomorphism
 $H^3(\cN_{X,\s}^s)\cong H^1(X)$ respects the polarisations.
\end{proposition}

\begin{proof}
Let $N=\cN_{X,\s}^s=\cN_X^s(\s;n,1,L_1,L_2)$ and $\overline
N=\cN_{X,\s}=\cN_X(\s;n,1,L_1,L_2)$. By (\ref{eqn:lll}), we can
assume $L_1=L_0$ and $L_2=\cO$, with $d=\deg(L_0)=d_1-nd_2$. As
$H^2(N)\cong \ZZ\oplus \ZZ$, we have that $\Pic(N)\cong
\ZZ\oplus\ZZ$. Fix a basis $H_1,H_2$ for $\Pic(N)$. For $\s$
rational, $\overline N$ is naturally polarised, that is, there are
$a,b\in \ZZ$ such that $H=aH_1+bH_2$ is a (primitive) polarisation
of $N$. Take a generic $(k-3)$-fold hyperplane intersection
$Z\subset N$. For $\s$ non-critical, $N$ is projective and smooth,
so $Z$ is smooth. For $\s$ critical, $Z$ is smooth as the
codimension of the singular locus $\overline N-N$ is at least $4$.
Now consider the polarisation
  \begin{eqnarray*}
  H^3(N)\ox H^3(N)  &\too& \ZZ\, ,\\
  \beta_1\ox \beta_2 &\mapsto & \la \beta_1\cup\beta_2, [Z]\ra\,.
  \end{eqnarray*}
This is a polarisation since $\codim (\overline N-N)\geq 5$, which
is proved as in Proposition \ref{prop:polarisation-MX}.

Again let $\pi:\cX\to \cT$ be the family of curves of genus $g$
with no automorphisms, and  consider the universal Jacobian
$q:\cJ^d\to \cT$. Then there is a universal moduli space
 $$
 p:\cN_{\cT,\s}=\cN_\cT (\s;n,1,d,\cO) \to \cJ^d
 $$
which puts over any $(X,L_0)\in \cJ^d$ the moduli space
$\cN_{X}^s(\s; n,1,L,\cO)$. There is a map $\cN_{\cT,\s_M^-}\to
\cM_\cT(n-1,d)$ defined on an open subset whose complement has
codimension at least two. Pulling back the relative ample
generator of $\cM_\cT(n-1,d)\to \cJ^d$, we get an element $H_2$
well defined in the family, $H_2\in \Pic(N)$. As $p$ is a
projective bundle (off a subset of codimension at least two),
there is another element $H_1$ which is well defined in the
family, $H_1\in \Pic(N)$. The construction of the flips can be
done in families, so $\Pic(N)\cong \ZZ[H_1]\oplus \ZZ[H_2]$ with
$H_1,H_2$ defined in families.

Let $\s$ be non-critical. Then $H+\epsilon_1H_1+\epsilon_2 H_2$ is
also a polarisation, for small rational $\epsilon_1,\epsilon_2$.
It is well-defined globally for the family. So the result
\cite[Lemma 8.1.1]{Ara-Sas} gives us a rational number
$m=m(\epsilon)$ so that
 $$
 \beta_1\cup \beta_2 \cup (H+\epsilon_1H_1+\epsilon_2 H_2)^{k-3}=
 m \, \theta (\beta_1,\beta_2)\,\quad \forall \beta_1,\beta_2\in H^3(N).
 $$
This implies that $\beta_1\cup \beta_2 \cup H_1^a \cup
H_2^{k-3-a}= m_a \, \theta$, for some $m_a\in \QQ$, for any $0\leq
a\leq k-3$. The conclusion is that for all possible polarisations
$\theta'$ of $N$ defined in families, we get that $\theta'$ is a
multiple of $\theta$.

If $\s$ is critical, then if there is only one polarisation for
$\cN_{X,\s}^s$, there is nothing to prove. If there are several,
then the ample cone contains an open set. So we can  work as above
to prove that all of the possible polarisations of $\cN_{X,\s}^s$
give the same polarisation (up to multiples) for
$H^3(\cN_{X,\s}^s)$.
\end{proof}

\section{The case of non-fixed determinant}
\label{sec:non-fixed-det}

In this section we shall prove the Torelli theorem for the moduli
spaces of pairs and bundles with non-fixed determinant, that is,
Corollary \ref{cor:Torelli-non-fixed-det}.

We shall use the following lemma.

\begin{lemma} \label{lem:fibers}
  Let $X$ be a projective connected variety and $f:X\to Y$ a map
  to another (quasi-projective) variety such that $f^*:H^k(Y)\to H^k(X)$ is zero
  for all $k>0$. Then $f$ is constant.
\end{lemma}

\begin{proof}
  Substituting $X$ by an
  irreducible component, we can assume $X$ is irreducible. Note that if
  $f$ is constant on each irreducible component, then
  it is constant, by the connectedness of $X$.

Let $d$ be the dimension of a generic fiber of $f$, and consider a
generic $d$-fold hyperplane intersection $Z\subset X$, which is
transverse to the generic fiber. Then $f|_Z:Z \to Y$ is a proper
and generically finite map. Therefore $f|_Z:Z\to f(Z)$ is of
finite degree $N>0$, and $f(Z)\subset Y$ is a closed subvariety.
Therefore $H^{2t}(f(Z))\cong \ZZ \to H^{2t}(Z)\cong \ZZ$, $t=\dim
Z$, is multiplication by $N$. If $t>0$, the assumption of the
lemma implies that $H^{2t}(Y)\to H^{2t}(f(Z))$ should be zero. But
this is impossible, since a generic $t$-fold hyperplane
intersection in $Y$ maps to a non-zero element in $H^{2t}(f(Z))$.
Therefore $t=0$, i.e. $f$ is constant.
\end{proof}

\noindent \emph{Proof of Corollary
\ref{cor:Torelli-non-fixed-det}.\/} Let $\frM_X(\tau;n,d)$ be the
moduli space of $\tau$-polystable pairs of rank $n$ and degree
$d$. There is a determinant map
 $$
 \mathrm{det} :\frM_X(\tau;n,d)\to \Jac^d X\, ,
 $$
sending $(E,\phi)\mapsto \det(E)$, whose fiber over $L_0$ is the
moduli space $\frM_X(n,L_0)$.

Assume that $F:\frM_{X}(\tau;n,d)\stackrel{\cong}{\too}
\frM_{X'}(\tau';n',d')$, for $X'$ another curve. Fix a line bundle
$L_0'$ on $X'$ of degree $d'$ and consider the composition
 $$
 f:\frM_{X'}(\tau';n',L_0') \inc \frM_{X'}(\tau';n',d') \cong
 \frM_{X}(\tau;n,d) \to \Jac^d X\, .
 $$
As $f^*:H^1(\Jac^d X)\to H^1(\frM_{X'}(\tau';n,L_0))=0$ is the
zero map, and $H^*(\Jac^d X)$ is generated by $H^1(\Jac^d X)$, we
have that the map $f^*:H^k(\Jac^d X)\to
H^k(\frM_{X'}(\tau';n,L_0))$ is zero for all $k>0$. Applying Lemma
\ref{lem:fibers}, we have that $f$ is constant. Therefore there
exists a line bundle $L_0$ on $X$ of degree $d$ such that $F$ maps
$M'=\frM_{X'}(\tau';n',L_0')$ to $M=\frM_{X}(\tau;n,L_0)$. Working
analogously, the map $F^{-1}$ maps $\frM_{X}(\tau;n,L_0)$ into
some fiber of the map $\det$, which must be $M'$. This implies
that $F|_{M'}:M'\to M$ is an isomorphism. Now we apply Corollary
\ref{cor:Torelli-fixed-det} to conclude that $X\cong X'$.

Suppose that $\tau'$ is a critical value and that there is an
isomorphism $F:\frM_{X}^s(\tau;n,d)\cong
\frM_{X'}^s(\tau';n',d')$. Then we have a map
 $$
  f:M'=\frM_{X'}^s(\tau';n',L_0') \inc \frM_{X'}^s(\tau';n',d') \cong
 \frM_{X}^s(\tau;n,d) \to \Jac^d X\, .
 $$
Now take any compactification $\bar M'$ of $M'$. So there is a
rational map $f:\bar M'\dashrightarrow \Jac^d X$. After
blowing-up, we have a compactificaction $\tilde M'$ of $M'$ and a
map $\tilde{f}:\tilde M'\to \Jac^d X$ which extends $f$. As
$H^1(M')=0$, we have that $H^1(\tilde M')=0$ as well. Using Lemma
\ref{lem:fibers}, we have that $f$ is a constant map. The rest of
the argument is as before.

The case of bundles is entirely analogous.\hfill $\Box$

\end{document}